\documentclass[a4paper,reqno]{amsart}

\usepackage[T1]{fontenc}
\usepackage[utf8x]{inputenc}
\usepackage[english]{babel}
\usepackage{yfonts}
\usepackage{dsfont}
\usepackage{amscd,amssymb,amsmath,amsthm,amsfonts}
\usepackage{mathtools}
\usepackage{accents}
\usepackage{tensor}
\usepackage{graphicx}
\usepackage{tikz}
\usetikzlibrary{calc,matrix,arrows,decorations.pathmorphing}
\usepackage{calc}
\usepackage[cal=boondox,scr=boondoxo]{mathalfa}
\usepackage{enumitem}
\usepackage{marginnote}
\usepackage{booktabs}
\usepackage{scalerel}[2016/12/29]
\usepackage{url}
\usepackage{contour}
\usepackage{ulem}
\usepackage{accents}
\usepackage{placeins}

\usepackage{hyperref}
\hypersetup{colorlinks=true,pageanchor=false,
linkcolor=blue,citecolor=red,urlcolor=red}
\usepackage[all]{hypcap}

\allowdisplaybreaks

\newtheorem{theorem}{Theorem}[section]

\newtheorem{proposition}[theorem]{Proposition}

\theoremstyle{definition}

\theoremstyle{remark}

\newcommand{\Z}{\mathbb{Z}}

\newcommand{\R}{\mathbb{R}}
\newcommand{\C}{\mathbb{C}}

\newcommand{\rmt}{\mathrm{t}}

\newcommand{\rmL}{\mathrm{L}}

\newcommand{\rmV}{\mathrm{V}}

\newcommand{\calC}{\mathcal{C}}
\newcommand{\calD}{\mathcal{D}}
\newcommand{\calE}{\mathcal{E}}

\newcommand{\calG}{\mathcal{G}}
\newcommand{\calH}{\mathcal{H}}

\newcommand{\calK}{\mathcal{K}}
\newcommand{\calL}{\mathcal{L}}

\newcommand{\calN}{\mathcal{N}}

\newcommand{\calR}{\mathcal{R}}

\newcommand{\calV}{\mathcal{V}}

\newcommand{\frakg}{\mathfrak{g}}

\renewcommand{\epsilon}{\varepsilon}
\renewcommand{\theta}{\vartheta}
\renewcommand{\phi}{\varphi}
\renewcommand{\Gamma}{\varGamma}
\renewcommand{\Sigma}{\varSigma}

\newcommand{\ad}{\mathrm{ad}}
\newcommand{\coad}{\mathrm{coad}}
\newcommand{\id}{\mathrm{id}}

\newcommand{\tr}{\mathrm{tr}}

\newcommand{\rad}{\operatorname{rad}}

\newcommand{\lev}{\smash{\stackrel{\leftarrow}{\mathrm{ev}}}}
\newcommand{\lcoev}{\smash{\stackrel{\longleftarrow}{\mathrm{coev}}}}
\newcommand{\rev}{\smash{\stackrel{\rightarrow}{\mathrm{ev}}}}
\newcommand{\rcoev}{\smash{\stackrel{\longrightarrow}{\mathrm{coev}}}}

\newcommand{\leqs}{\leqslant}
\newcommand{\geqs}{\geqslant}

\newcommand{\mods}[1]{\operatorname{\mathnormal{#1}-mod}}

\newcommand{\fsl}{\mathfrak{sl}}
\newcommand{\SL}{\mathrm{SL}}
\newcommand{\SO}{\mathrm{SO}}

\newcommand{\PGL}{\mathrm{PGL}}
\newcommand{\MCG}{\mathrm{Mod}}
\newcommand{\HH}{\mathrm{HH}}

\newcommand{\End}{\mathrm{End}}

\newcommand{\op}{\mathrm{op}}

\newcommand{\RT}{\mathrm{RT}}

\DeclareRobustCommand{\one}{\mathbin{\text{\includegraphics[height=\heightof{$\mathbf{1}$}]{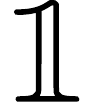}}}}

\newcommand{\Proj}{\mathrm{Proj}}

\makeatletter
\newcommand{\subalign}[1]{
  \vcenter{
    \Let@ \restore@math@cr \default@tag
    \baselineskip\fontdimen10 \scriptfont\tw@
    \advance\baselineskip\fontdimen12 \scriptfont\tw@
    \lineskip\thr@@\fontdimen8 \scriptfont\thr@@
    \lineskiplimit\lineskip
    \ialign{\hfil$\m@th\scriptstyle##$&$\m@th\scriptstyle{}##$\crcr
      #1\crcr
    }
  }
}
\makeatother

\def\clap#1{\hbox to 0pt{\hss#1\hss}}

\newcommand{\pic}[2][0]{\raisebox{-0.5\height + 2.5pt + #1pt}{\includegraphics{#2.pdf}}}

\newcommand\arxiv[2]{\href{https://arXiv.org/abs/#1}{\texttt{arXiv:\allowbreak #1} #2}}
\newcommand\doi[2]{\href{https://doi.org/#1}{#2}}

\contourlength{1pt}

\DeclareRobustCommand{\myuline}[1]{
 \ifmmode \text{\uline{$\phantom{#1}$}\llap{\contour{white}{$#1$}}}
 \else \uline{\phantom{#1}}\llap{\contour{white}{#1}} \fi
}

\begin{document}

\raggedbottom

\title{Modular Categories and TQFTs Beyond Semisimplicity}

\author[C. Blanchet]{Christian Blanchet} 
\address{Université de Paris, Sorbonne Université, CNRS, IMJ-PRG, F-75006 Paris, France}
\email{christian.blanchet@imj-prg.fr}

\author[M. De Renzi]{Marco De Renzi} 
\address{Institute of Mathematics, University of Zurich, Winterthurerstrasse 190, CH-8057 Zurich, Switzerland} 
\email{marco.derenzi@math.uzh.ch}

\begin{abstract}
 Vladimir Turaev discovered in the early years of quantum topology that the notion of modular category was an appropriate structure for building 3-di\-men\-sion\-al Topological Quantum Field Theories (TQFTs for short) containing invariants of links in 3-manifolds such as Witten--Reshetikhin--Turaev ones. In recent years, generalized notions of modular categories, which relax the semisimplicity requirement, have been successfully used to extend Turaev's construction to various non-semisimple settings. We report on these recent developments in the domain, showing the richness of Vladimir's lineage.
\end{abstract}

\maketitle
\setcounter{tocdepth}{3}

\section{Introduction}

After the work of Vaughan Jones \cite{J85} and Edward Witten \cite{W88,W89}, Vladimir Turaev played a key role in the foundation of quantum topology. Together with Nikolai Reshetikhin \cite{RT91}, they were the first to obtain a mathematically rigorous construction of Witten's invariants of 3-dimensional manifolds. Turaev showed that the relevant algebraic structure underlying their approach was that of a \textit{modular category}, a notion he himself identified and introduced for this purpose \cite{T94}. Starting from a modular category, one can build a Topological Quantum Field Theory (TQFT for short) and, in particular, invariants of links in 3-manifolds, and representations of mapping class groups of surfaces. The famous Witten--Reshetikhin--Turaev invariants can be derived from a family of modular categories arising as semisimple quotients of representation categories of the restricted quantum group of $\fsl_2$ at roots of unity. Around the same time, equivalent modular categories were obtained from the Kauffman bracket skein theory \cite{BHMV95}. On the one hand, the quantum group construction generalizes to other families of modular categories, corresponding to arbitrary simple Lie algebras $\frakg$. On the other hand, the skein construction produces families of modular categories associated with the structure of Hecke algebras of type A \cite{B98} and of Birman--Murakami--Wenzl algebras of type B,C, and D \cite{BB00}. The abundance of examples shows the relevance of Turaev's definition. However, in all these instances, only simple representations, which cover just a small part of the corresponding representation theory, are used.

Already within the first decade since the birth of quantum topology, constructions of invariants and mapping class group representations beyond the scope of semisimple modular categories emerged, thanks to the work of Mark Hennings \cite{H96} and Volodymyr Lyubashenko \cite{L94}. More recently, the algebraic theory of \textit{modified traces}, whose germs can be found in an early collaboration of Nathan Geer and Bertrand Patureau with Turaev \cite{GPT07}, served as a huge propeller for a new wave of non-semisimple quantum constructions. Through this technology, TQFTs and HQFTs\footnote{Acronym for Homotopy Quantum Field Theory, a notion introduced by Turaev.} were built for the first time out of non necessarily semisimple generalizations of modular categories. Our purpose here is to review the original construction, and to relate it to the latest achievements in the field.

\section{Historical overview}\label{S:history}

\subsection*{Definition of TQFT}

A popular axiomatization of the notion of TQFT was proposed by Micheal Atiyah \cite{A88}, who was inspired by a similar definition of Graeme Segal \cite{S88} for the concept of a Conformal Field Theory (CFT for short). Paraphrasing Atiyah, a TQFT in dimension 3 can be defined as a symmetric monoidal functor from a category of 3-dimensional cobordisms to the category of vector spaces. In this definition, cobordism categories are naturally equipped with the symmetric monoidal structure induced by disjoint union. In other words, a TQFT $V$ associates with every closed oriented surface $\Sigma$ a vector space $V(\Sigma)$ called the \textit{state space} of $\Sigma$, and it associates with every cobordism $M$ from $\Sigma$ to $\Sigma'$ a linear map $V(M) : V(\Sigma) \to V(\Sigma')$. At the same time, a TQFT translates the topological operations of gluing and disjoint union into the algebraic operations of composition and tensor product, respectively. The use of an indefinite article in this definition, when we talk about \textit{a} category of cobordisms, as opposed to \textit{the} category of cobordisms, is due to the fact that cobordisms are typically allowed to carry additional structure, which may vary according to the different constructions we might want to consider. All the resulting notions are gathered under the umbrella term ``TQFT''.

\subsection*{Semisimple constructions}

In his seminal work \cite{W88}, Witten defined topological invariants of closed oriented 3-manifolds in terms of Chern--Simons gauge theory, and he predicted that these invariants could be understood as part of a TQFT. Furthermore, he explained in \cite{W89} that his construction actually provides a quantum field theoretical interpretation for the mysterious Jones polynomial \cite{J85}, a topological invariant of knots whose geometric content remains to this day quite elusive. Witten's approach uses tools, such as Feynman's path integral, which have not been formalized on a mathematical level yet. Shortly after, Resthetikhin and Turaev developed an alternative method for the construction of a large family of topological invariants of closed oriented 3-manifolds containing Witten's invariants, and based on a semisimple version of the representation theory of certain finite-dimensional ribbon Hopf algebras called \textit{modular Hopf algebras} \cite{RT91}. Turaev later introduced his notion of modular category by distilling and axiomatizing the properties of simple modules over these Hopf algebras which play a key role in the construction. In this very general framework, Turaev managed to complete Witten's program and to rigorously define for the first time TQFT functors inducing projective representations of mapping class groups of closed oriented surfaces \cite{T94}. Although this mature version of the theory is due to Turaev alone, it is still common, for a modular category $\calC$, to refer to the resulting TQFT $\rmV_\calC$ as a Reshetikhin--Turaev TQFT, and to denote the underlying topological invariant as $\RT_\calC$. Turaev's approach crucially exploits the semisimplicity assumption which is hardcoded into the original definition of modular categories. However, none of the quantum groups featured in the main family of examples naturally meet this condition. This means their categories of finite-dimensional representations have to undergo a quotient process which sacrifices algebraic information.

The main instances of finite-dimensional ribbon Hopf algebras inducing TQFTs through Turaev's approach are provided by so-called \textit{small} and \textit{restricted} quantum groups. In Section \ref{S:sl_2_sem-sim}, we discuss the simplest possible case, that of $\bar{U}_q \fsl_2$. For this particular family of concrete examples, whose corresponding topological invariants are spelled out in detail in \cite{KM91}, alternative constructions exist. In \cite{L93}, Lickorish obtained the same 3-manifold invariants relying exclusively on combinatorial techniques based on Temperley--Lieb algebras and Kauffman bracket polynomials. These invariants were then extended to TQFTs by Blanchet, Habegger, Masbaum, and Vogel in \cite{BHMV95} through a general procedure called the \textit{universal construction}. This approach, which identifies state spaces of closed surfaces with quotients of skein modules of handlebodies, provides a very natural method for the definition of functorial extensions of closed 3-manifold invariants to categories of cobordisms, and it has since been successfully applied in a wide range of different settings.

\subsection*{Early non-semisimple constructions}

Reshetikhin and Turaev's work immediately sparked a wide interest in the newly born field of quantum topology, and some natural questions quickly emerged. First of all, is it possible to figure out the relevant combinatorics for elements of ribbon Hopf algebras arising from the construction of \cite{RT90} without relying on their representation theory? Answering this question required learning how to deal with Hopf algebras themselves, which, in the relevant family of concrete examples, are not semisimple, and thus are not completely determined by their simple modules. Therefore, another natural question was: can we relax the semisimplicity requirement of Turaev? Lawrence explained in \cite{La90} how to construct so-called \textit{universal link invariants} using directly quantum groups, and thus bypassing their representation theory. These ideas were expanded by Hennings \cite{H96}, who upgraded Lawrence's construction to obtain invariants of closed oriented 3-manifolds starting from finite-dimensional factorizable\footnote{Factorizability is the appropriate non-degeneracy condition for ribbon Hopf algebras in the non-semisimple setting, see Section \ref{S:L-mod_cat}.} ribbon Hopf algebras. Since semisimplicity was no longer required, this constituted the first non-semisimple generalization of Reshetikhin--Turaev invariants, and the start of non-semisimple quantum topology. Hennings' result was reformulated by Kauffman and Radford avoiding the use of orientations for framed links \cite{KR94}. Both of these constructions were extended by Lyubashenko, who proposed a more general definition of modular category with respect to Turaev's original one, and used it to build topological invariants of 3-manifolds and projective representations of mapping class groups of surfaces \cite{L94}. For such a modular category $\calC$, which is allowed to be non-semsimple, we denote with $\rmL_\calC$ the corresponding Lyubashenko invariant of closed oriented 3-manifolds. The main families of examples are provided by categories of finite-dimensional representations of small quantum groups, without quotient operations. In Section \ref{S:sl_2_non-sem-sim}, we discuss the case of $\bar{U}_q \fsl_2$. 

Lyubashenko's work raised again some natural questions:
\begin{enumerate}
 \item Since this construction produces all the desired byproducts one expects to obtain from a TQFT, is there actually a TQFT lying behind them? 
 \item Given that the loss of algebraic information is successfully avoided, do we get finer topological invariants and mapping class group representations?
\end{enumerate}  
A few results which appeared in the early stages of the development of this new approach seemed to suggest negative answers to both questions. More precisely, it was first shown by Ohtsuki in a special case \cite{Oh95}, and later by Kerler in general \cite{K96b}, that if $\calC$ is a non-semisimple modular category, and if $M$ is a closed oriented 3-manifold with first Betti number $b_1(M) > 0$, then
\[
 \rmL_\calC(M) = 0.
\]
This is a deep obstruction for the definition of TQFT extensions. Indeed, it is a standard consequence of symmetric monoidality that the dimension of the state space $\rmV_\calC(\Sigma)$ of a closed oriented surface $\Sigma$ for a TQFT $\rmV_\calC$ extending $\rmL_\calC$ is
\[
 \dim_\Bbbk \rmV_\calC(\Sigma) = \rmL_\calC(\Sigma \times S^1).
\]
This means that Lyubashenko's invariant associated with a non-semisimple modular category $\calC$ cannot be extended to a TQFT, at least not in the sense of Atiyah. For years, the most involved achievement in a non-semisimple setting was a result of Lyubashenko and Kerler \cite{KL01} which produces a partial version of an Extended TQFT (ETQFT for short) satisfying only weak monoidality properties.\footnote{Very briefly, an ETQFT is a 2-categorical analogue of a TQFT, i.e. a symmetric monoidal 2-functor from a 2-category of cobordisms to a 2-category of linear categories (several choises are possible also for the target). Kerler and Lyubashenko's 2-functors are only defined for connected surfaces, with symmetric monoidal structure induced by connected sum instead of disjoint union.} Furthermore, it was later shown that, at least for small quantum $\fsl_2$, Hennings and Lyubashenko's invariants can be expressed as functions of Witten--Reshetikhin--Turaev ones \cite{CKS07}. More precisely, if $\calC$ is the category of finite-dimensional representations of $\bar{U}_q \fsl_2$, and if $\bar{\calC}$ is the semisimple quotient of $\calC$ inducing Witten's invariants, then for every 3-manifold $M$ we have
\[
 \rmL_\calC(M) = h_1(M) \RT_{\bar{\calC}}(M),
\]
where $h_1(M) = \left| H_1(M) \right|$ if $b_1(M) = 0$, and $h_1(M) = 0$ if $b_1(M) > 0$. Because of these hurdles, these deep works were underestimated for quite some time. This is not to say that the constructions were forgotten, as the number of publications containing reviews of Hennings' approach witnesses \cite{K96a,V03,H05}. However, a feeling that no additional topological information could be accessed through non-semisimple generalizations of Reshetikhin and Turaev's work started to spread.

\subsection*{Renormalized non-semisimple constructions}\label{SS:renormalized_non-semisimple}

A key ingredient that was missing from these early attempts at non-semisimple constructions was the theory of \textit{modified traces} developed by Geer and Patureau together with many collaborators, among which Turaev \cite{GPT07}, Kujawa \cite{GKP10}, and Virelizier \cite{GPV11}. Indeed, one way to explain the problems we face when trying to generalize Turaev's TQFT construction is that the categorical trace operation, which associates scalars to endomorphisms in a ribbon category $\calC$, becomes highly degenerate as soon as $\calC$ is non-semisimple. For instance, in this case the categorical trace of every endomorphism of any projective object is zero. Since this operation is pervasive in all quantum topological constructions, invariants of closed oriented 3-manifolds are bound to vanish a lot in non-semisimple settings. The idea behind \textit{renormalization} consists in replacing this highly degenerate operation with a non-degenerate analogue, called modified trace, which in general is only defined on proper tensor ideals of $\calC$. This approach was pioneered by Costantino, Geer, and Patureau \cite{CGP12} in a different non-semisimple setting with respect to Lyubashenko's. Indeed, an alternative (and also potentially non-semisimple) generalization of Turaev's notion of modular category, called \textit{relative modular category}, exists. Relative modular categories are so different from Lyubashenko's modular categories that the two notions intersect exclusively along Turaev's modular categories. In other words, if a category $\calC$ is both modular in the sense of Lyubashenko and relative modular, then it is modular in the sense of Turaev. The main examples of relative modular categories are associated with a different family of quantum groups, called \textit{unrolled quantum groups}. Starting from a relative modular category, one can define a topological invariant of closed oriented 3-manifolds decorated with so-called \textit{admissible cohomology classes} which is called the \textit{CGP invariant}. Focusing on the unrolled quantum group $U^H_q \fsl_2$, this construction was extended for the first time to non-semisimple graded TQFTs in \cite{BCGP14}. The adaptation of this result to arbitrary relative modular categories, as well as a complete extension to graded ETQFTs in general, is presented in \cite{D17}.

Following the success of these constructions, renormalization was soon applied to the theory of Hennings and Lyubashenko. This was done at first without any mention to modified traces, and thus independently of the work of Geer and Patureau, by Murakami \cite{M13}. The resulting invariants of so-called \textit{admissible knots} in 3-manifolds were extended to invariants of \textit{admissible links} in 3-manifolds, and reformulated in terms of modified traces, by Beliakova, Blanchet, and Geer \cite{BBG17}. Both of these works deal exclusively with the case of the restricted quantum group $\bar{U}_q \fsl_2$. A generalization of these constructions to arbitrary finite-dimensional factorizable ribbon Hopf algebras, as well as corresponding TQFT extensions, was first obtained in \cite{DGP17}. The adaptation of this theory to the general framework of Lyubashenko was later achieved in \cite{DGGPR19}, where TQFTs are constructed starting from arbitrary, potentially non-semisimple modular categories. This result recovers both Turaev's TQFTs, in the semisimple case, and Lyubashenko's mapping class group representations, in general. It should be remarked, however, that the renormalization process changes the underlying 3-manifold invariants drastically. In particular, a complete answer to question $(i)$ is that Lyubashenko's invariants cannot be extended to TQFTs, but that his mapping class group representations can.

The invariants and representations arising from the CGP and the renormalized Lyubashenko constructions are related to each other, see \cite{DGP18} for a discussion of the quantum group case. Both theories exhibit behaviors which mark a sharp contrast with respect to the semisimple case. First of all, the CGP invariant associated with $U^H_q \fsl_2$ for $q=i$ recovers the abelian Reidemeister torsion, which completely classifies lens spaces \cite{BCGP14}. This could not be achieved by any Witten--Reshetikhin--Turaev invariant. Furthermore, the action of Dehn twists along simple closed curves on states spaces associated with both $U^H_q \fsl_2$ and $\bar{U}_q \fsl_2$ have infinite order \cite{BCGP14,DGGPR20}. This is also something that never happens in the semisimple case, and which provides a positive answer to question $(ii)$. We point out that, to the best of our knowledge, these remarkable properties of Lyubashenko's mapping class group representations had never been noticed before the completion of the TQFT program.

We finish this section by stressing an important feature of non-semisimple constructions. The admissiblity conditions we mentioned in several places have to do with the fact that modified traces are generally defined only on proper tensor ideals. This means that, if we want to add them to our toolbox, we need decorations of 3-manifolds to meet certain requirements. Typically, decorations carry labels given by objects of a category $\calC$, and we need at least one of these labels to belong to the tensor ideal supporting the modified trace. It should be noted that these constraints also affect cobordism categories. Indeed, certain decorated cobordisms have to be removed from sets of morphisms, and this operation usually breaks the rigidity of a few objects. In other words, certain decorated surfaces are not dualizable in these so-called \textit{admissible decorated cobordism categories}. However, the symmetric monoidal structure induced by disjoint union remains intact, and so the resulting TQFTs are indeed symmetric monoidal in the sense of Atiyah.

\section{Turaev's notion of modular category}\label{S:T-mod_cat}

Using the language of \cite{EGNO15}, a \textit{T-modular category} (T for Turaev) is a ribbon fusion category whose S-matrix is invertible. The definition originally appeared in \cite[Sec.~II.1.2]{T94} in a slightly more general form, but the version reported here has since become rather standard (compare with \cite{BK01,TV17}). In order to better grasp this notion, let us proceed to unpack the terminology. 

First of all, a \textit{ribbon} category $\calC$ is a braided rigid monoidal category with a twist. This means a ribbon category comes equipped with:
\begin{enumerate}
 \item a tensor product $\otimes : \calC \times \calC \to \calC$ and a tensor unit $\one \in \calC$;
 \item duality morphisms $\lev_X : X^* \otimes X \to \one$, $\lcoev_X : \one \to X \otimes X^*$, and $\rev_X : X \otimes X^* \to \one$, $\rcoev_X : \one \to X^* \otimes X$ for every $X \in \calC$;
 \item braiding morphisms $c_{X,Y} : X \otimes Y \to Y \otimes X$ for all $X,Y \in \calC$;
 \item twist morphisms $\theta_X : X \to X$ for every $X \in \calC$.
\end{enumerate}
All these structure data are subject to several axioms, which can be found explicilty in \cite[Sec.~2.1, 2.10, 8.1, 8.10]{EGNO15}, but which can also be understood efficiently using a more visual approach. Indeed, morphisms of a ribbon category can be represented using an extension of \textit{Penrose graphical notation} and \textit{string diagrams}. Structure morphisms listed above are represented by elementary diagrams
\begin{align*}
 \id_X &= \pic{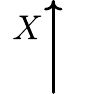} &
 \id_{X^*} &= \pic{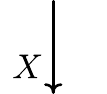} \\[2pt]
 \lev_X &= \pic{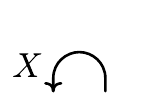} &
 \lcoev_X &= \pic{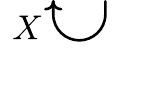} \\[2pt]
 \rev_X &= \pic{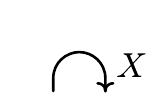} &
 \rcoev_X &= \pic{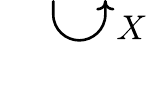} \\[2pt]
 c_{X,Y} &= \pic{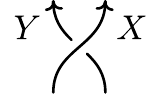} &
 \theta_X &= \pic{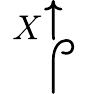}
\end{align*}
In general, a morphism $f : X_1 \otimes \ldots \otimes X_m \to Y_1 \otimes \ldots \otimes Y_n$ which is not expressed as a combination of structure morphisms is simply placed inside a box 
\begin{equation}\label{E:coupon}
 f = \pic{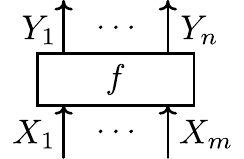}
\end{equation}
These diagrams are read from bottom to top.\footnote{Remark that the convention used here for orientations is opposite to the one used by Turaev, but also widely used.} Composition is given by vertical stacking, while tensor product corresponds to horizontal juxtaposition. Then, the axioms of a ribbon category translate precisely to the invariance of this diagrammatic calculus under isotopy and framed Reidemeister moves.

Next, let us fix an algebraically closed\footnote{This condition is not necessary, but convenient.} field $\Bbbk$. Throughout this paper, we will use the term ``linear'' as a shorthand for ``$\Bbbk$-linear''. Then, let us consider a linear category $\calC$, which is a category whose sets of morphisms carry vector space structures making compositions into bilinear maps. We say that $\calC$ is \textit{finite} if it is equivalent to the category of finite-dimensional representations of a finite-dimensional algebra. This definition is rather implicit, as no algebra is attached to $\calC$, but it can be reformulated in more explicit terms. In fact, as explained in \cite[Def.~1.8.6]{EGNO15}, $\calC$ is finite if and only if it is abelian and:
\begin{enumerate}
 \item it has a finite number of isomorphism classes of simple objects;
 \item it has enough projectives;
 \item its objects have finite lenght; 
 \item its morphism spaces have finite dimension.
\end{enumerate}
A \textit{fusion category} is a finite semisimple monoidal linear category. Therefore, it is this assumption that encodes the semisimplicity of T-modular categories. 

Finally, for every ribbon fusion category we can define an \textit{S-matrix} given by 
\[
 S(\calC) := \left( \pic{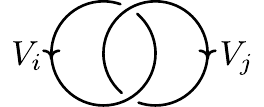} \right)_{i,j \in I}
\]
where $\Gamma(\calC) = \left\{ V_i \mid i \in I \right\}$ is a set of representatives of isomorphism classes of simple objects of $\calC$. We say a ribbon fusion category $\calC$ is T-modular if $S(\calC)$ is invertible. This requirement can be understood as a non-degeneracy condition for the braiding of a ribbon fusion category.

We finish this section with an additional definition. The \textit{Müger center} $M(\calC)$ of a ribbon category $\calC$ is the full subcategory of so-called \textit{transparent} objects of $\calC$, that is, $X \in M(\calC)$ if $c_{Y,X} \circ c_{X,Y} = \id_{X \otimes Y}$ for every $Y \in \calC$. The Müger center of a ribbon category $\calC$ is said to be trivial if every transparent object of $\calC$ is isomorphic to a direct sum of copies of the tensor unit $\one$. Then, a ribbon fusion category $\calC$ has trivial Müger center if and only if $S(\calC)$ is invertible \cite[Prop.~1.1]{B00}. This equivalent non-degeneracy condition for the braiding of $\calC$ will be useful later on.

\section{Semisimple quantum invariants and TQFTs}\label{S:TQFT_sem-sim}

The diagrammatic interpretation for morphisms of ribbon categories which was discussed earlier highlights the fundamental interplay between algebra and topology arising in this context. In fact, it is a result of Turaev that certain topological gadgets called \textit{ribbon graphs} provide the relevant topological structure underlying Penrose graphical notation \cite[Thm.~I.2.5]{T94}. Indeed, isotopy classes of ribbon graphs can be organised as morphisms of a category $\calR_\calC$, and the graphical calculus introduced above can be implemented by a functor
\[
 F_\calC : \calR_\calC \to \calC
\]
called the \textit{Reshetikhin--Turaev functor}. This technology plays a fundamental role in every construction discussed in this survey.

Starting from a T-modular category $\calC$, Turaev showed how to construct a topological invariant $\RT_\calC$ of closed oriented 3-manifolds decorated with closed\footnote{A ribbon graph is \textit{closed} if it has no boundary vertices, and thus represents an endomorphism of the tensor unit $\one$ of $\calC$ under the Reshetikhin--Turaev functor $F_\calC$.} ribbon graphs. By their very definition, ribbon graphs have a rather hybrid nature, between topology and algebra. This of course can be seen as a strenght of the theory, which allows for great flexibility. However, for the benefit of those who are looking for a purely topological point of view, the construction can be understood as producing, in particular, a family $\{ \RT_V \}_{V \in \calC}$ of topological invariants of framed oriented links inside 3-manifolds, where each invariant $\RT_V$ is obtained by choosing the object $V \in \calC$ as a label for every component of the link. The definition of these topological invariants is based on surgery presentations, which allow us to interpret framed links in $S^3$ as closed oriented 3-manifolds. The main idea consists in choosing orientations and labels for a surgery presentation, and then using the Reshetikhin--Turaev functor $F_\calC$ in order to extract a scalar. The choice needs to be made in such a way that the resulting scalar is invariant under orientation reversals and Kirby moves. Turaev came up with such a label, which is given by a formal linear combination of objects of $\calC$ called the \textit{Kirby color}, and given by
\[
 \Omega := \sum_{i \in I} \dim_\calC(V_i) \cdot V_i,
\]
where again $\Gamma(\calC) = \left\{ V_i \mid i \in I \right\}$ is a set of representatives of isomorphism classes of simple objects of $\calC$, and where
\[
 \dim_\calC(X) = \tr_\calC(\id_X) = \rev_X \circ \lcoev_X
\]
is the \textit{categorical dimension} of an object $X \in \calC$, which coincides with the \textit{categorical trace} of its identity endomorphism. Labeling components of a framed oriented link using the Kirby color is a shortcut that has to be understood through multilinear expansions. More precisely, if $L$ is a framed oriented link featuring a component $K$ labeled by $\Omega$, and if $L_i$ is obtained from $L$ by replacing the label of $K$ with $V_i \in \Gamma(\calC)$, then we set
\[
 F_\calC(L) := \sum_{i \in I} \dim_\calC(V_i) F_\calC(L_i).
\]
This oriented framed link invariant is already preserved under Kirby II moves, as the terminology might suggest. In order to obtain invariance under Kirby I moves, we need a normalization coefficient. This is defined in terms of the so-called \textit{stabilization coefficients}
\[
 \Delta_+ := F_\calC \left( \pic{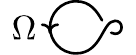} \right) \in \Bbbk, \qquad
 \Delta_- := F_\calC \left( \pic{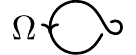} \right) \in \Bbbk.
\]
As a consequence of T-modularity, their product $\Delta_+ \Delta_-$ is invertible, and so we can choose two parameters $\calD,\delta \in \Bbbk^\times$ satisfying
\[
 \calD^2 = \Delta_+ \Delta_-, \qquad 
 \delta = \frac{\Delta_+}{\calD} = \frac{\calD}{\Delta_-}.
\]
The definition of the topological invariant $\RT_\calC$ depends on the choice of $\calD$, which uniquely determines $\delta$. If $M$ is a closed oriented 3-manifold, if $T \subset M$ is a closed ribbon graph, and if $L \subset S^3$ is a surgery presentation of $M$, which we assume having $\ell$ components and signature $\sigma$, and which we equip with an arbitrary orientation and the label $\Omega$ on every component, then
\[
 \RT_\calC(M,T) := \calD^{-1-\ell} \delta^{-\sigma} F_\calC(L \cup T)
\]
is a topological invariant of the pair $(M,T)$.

The universal construction of \cite{BHMV95} extends the invariant $\RT_\calC$ to a dual pair of TQFTs given by a covariant functor $\rmV_\calC$ and a contravariant one $\rmV'_\calC$. This means that, for a closed oriented surface $\Sigma$, we obtain dual vector spaces $\rmV_\calC(\Sigma)$ and $\rmV'_\calC(\Sigma)$ supporting respectively left and right actions of decorated cobordisms. If $\bar{\Sigma}$ is obtained from $\Sigma$ by reversing its orientation, we have a natural isomorphism $\rmV'_\calC(\Sigma) \cong \rmV_\calC(\bar{\Sigma})$. This means we can simply focus on the covariant TQFT $\rmV_\calC$. If $\Sigma_g$ is a closed oriented surface of genus $g$, then
\begin{equation}\label{E:st_sp_sem-sim}
 \rmV_\calC(\Sigma_g) \cong \bigoplus_{i_1,\ldots,i_g \in I} \calC \left( \one,\bigotimes_{j=1}^g V_{i_j} \otimes V_{i_j}^* \right).
\end{equation}
Since $\calC$ is semisimple, this vector space can be rewritten, using fusion rules, in terms of labelings of trivalent graphs, with different graphs determining different decompositions. If $\calG$ is an oriented trivalent graph, whose sets of vertices and egdes are denoted $\calV$ and $\calE$ respectively, then, for every vertex $v \in \calV$, we denote with $v_-$ the set of incoming germs of edges with target $v$, and we denote with $v_+$ the set of outgoing germs of edges with source $v$. A fundamental labeling of $\calE$ is a function $i : \calE \rightarrow I$, and for every $v \in \calV$ the morphism space
\[
 \calC(i,v) := \calC \left( \bigotimes_{e \in v_-} V_{i(e)},\bigotimes_{e \in v_+} V_{i(e)} \right)
\]
is called the \textit{multiplicity module} of $i$ in $v$. Then, let us consider the embedded oriented trivalent graph
\begin{equation}\label{E:triv_graph}
 \calG_g := \pic{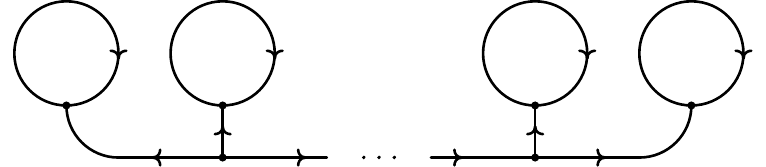} \subset \R^3, 
\end{equation}
and let us denote with $\calV_g$ and $\calE_g$ its sets of vertices and edges respectively. We should think of $\Sigma_g$ as being identified with the boundary of a regular neighborhood of $\calG_g$. Then, the state space $\rmV_\calC(\Sigma_g)$ decomposes as a direct sum of tensor products of multiplicity modules
\begin{equation}\label{E:E:st_sp_sem-sim_dec}
 \rmV_\calC(\Sigma_g) \cong \bigoplus_{i : \calE_g \to I} \bigotimes_{v \in \calV_g} \calC(i,v).
\end{equation}

\section{T-modular categories from small quantum \texorpdfstring{$\mathfrak{sl}_2$}{sl(2)}}\label{S:sl_2_sem-sim}

The Witten--Reshetikhin--Turaev TQFT was first constructed from the restricted quantum group of $\fsl_2$ at even roots of unity. A careful exposition can be found in \cite{T94}, while a popular alternative construction, based on Jones--Kauffman skein theory, can be found in \cite{BHMV95}. This skein construction also produces TQFTs when the quantum parameter is an odd root of unity. Referring to Witten's gauge picture, these are called $\SO(3)$-TQFTs. We choose this family as a demonstrative example, and develop in this section the construction of the underlying modular categories, corresponding to the so-called small quantum group of $\fsl_2$ at odd roots of unity.

We start by fixing an odd integer $3 \leqs r \in \Z$, and we consider the primitive $r$th root of unity $q = e^{\frac{2 \pi i}{r}}$. For every integer $k \geqs 0$ we introduce the notation 
\[
 \{ k \} := q^k - q^{-k},
 \quad \{ k \}' := q^k + q^{-k},
 \quad [k] := \frac{\{ k \}}{\{ 1 \}},
 \quad [k]! := \prod_{j=1}^k [j].
\]
The small quantum group $\bar{U}_q \fsl_2$, first defined by Lusztig in \cite{Lu90}, can be constructed as the $\C$-algebra with generators $\{ E,F,K \}$ and relations
\begin{gather*}
 E^r = F^r = 0, \qquad K^r = 1, \\*
 K E K^{-1} = q^2 E, \qquad K F K^{-1} = q^{-2} F, \qquad [E,F] = \frac{K - K^{-1}}{q-q^{-1}},
\end{gather*}
and with Hopf algebra structure obtained by setting
\begin{align*}
 \Delta(E) &= E \otimes K + 1 \otimes E, & \varepsilon(E) &= 0, & S(E) &= -E K^{-1}, \\*
 \Delta(F) &= K^{-1} \otimes F + F \otimes 1, & \varepsilon(F) &= 0, & S(F) &= - K F, \\*
 \Delta(K) &= K \otimes K, & \varepsilon(K) &= 1, & S(K) &= K^{-1}.
\end{align*}
Remark that Lusztig considers the opposite coproduct, while we are using the one of Kassel \cite[Sec.~VII.1]{K95}. A basis of $\bar{U}_q \fsl_2$ is given by
\[
 \left\{ E^a F^b K^c \mid \ 0 \leqs a,b,c \leqs r - 1 \right\},
\]
as proved in \cite[Thm.~5.6]{Lu90}. Furthermore, $\bar{U}_q \fsl_2$ supports the structure of a ribbon Hopf algebra. Indeed, an \textit{R-matrix} $R \in \bar{U}_q \fsl_2 \otimes \bar{U}_q \fsl_2$ is given by
\begin{align*}
 R &:= \frac{1}{r} \sum_{a,b,c=0}^{r-1} \frac{\{ 1 \}^a}{[a]!}
 q^{\frac{a(a-1)}{2} - 2bc} K^b E^a \otimes K^c F^a, \\*
 R^{-1} &\phantom{:}= \frac{1}{r} \sum_{a,b,c=0}^{r-1} \frac{\{ -1 \}^a}{[a]!}
 q^{-\frac{a(a-1)}{2} + 2bc} E^a K^b \otimes F^a K^c,
\end{align*}
while a \textit{ribbon element} $v \in \bar{U}_q \fsl_2$ is given by
\begin{align*}
 v &:= \frac{i^{\frac{r-1}{2}}}{\sqrt{r}} \sum_{a,b = 0}^{r - 1} \frac{\{ -1 \}^a}{[a]!} q^{-\frac{a(a-1)}{2} + \frac{(r+1)(a-b-1)^2}{2}} F^a K^b E^a, \\*
 v^{-1} &\phantom{:}= \frac{i^{-\frac{r-1}{2}}}{\sqrt{r}} \sum_{a,b = 0}^{r - 1} \frac{\{ 1 \}^a}{[a]!} q^{\frac{a(a-1)}{2} + \frac{(r-1)(a+b-1)^2}{2}} F^a K^b E^a.
\end{align*}
An explicit formula for $R$ first appeared in \cite[Sec.~3.5]{R93}, see also \cite[Sec.~4.5]{Oh02} for a different one. The expression for $v$ reported here can be obtained by adapting the proof of \cite[Thm.~4.1.1]{FGST05}. A \textit{pivotal element} $g \in \bar{U}_q \fsl_2$ is given by $g := K$. It is proved in \cite[Prop.~XIV.6.5]{K95} that this choice defines a compatible ribbon structure, in the sense that $v = ug^{-1}$ for the \textit{Drinfeld element} $u \in \bar{U}_q \fsl_2$ given by $u := S(R'')R'$, where the notation $R = R' \otimes R''$ hides a sum. Thus the category $\calC = \mods{\bar{U}_q \fsl_2}$ of finite-dimensional left $\bar{U}_q \fsl_2$-modules is a ribbon category.

For every integer $0 \leqs n \leqs r-1$ we denote with $V_n$ the $\bar{U}_q \fsl_2$-module with basis 
\[
 \{ a_i^n \in V_n \mid 0 \leqs i \leqs n \}
\]
and action given, for all integers $0 \leqs i \leqs n$, by
\begin{equation*}
 \begin{split}
  K \cdot a_i^n &= q^{n-2i} a_i^n, \\
  E \cdot a_i^n &= [i][n-i+1] a_{i-1}^n, \\
  F \cdot a_i^n &= a_{i+1}^n,
 \end{split}
\end{equation*}
where $a_{-1}^n := a_{n+1}^n := 0$. These $\bar{U}_q \fsl_2$-modules are all simple and self-dual, and every simple object of $\calC$ is isomorphic to $V_n$ for some $0 \leqs n \leqs r-1$. Their categorical dimension is
\[
 \dim_\calC(V_n) = \tr_\calC(\id_{V_n}) = [n+1].
\]
We will see later that the category $\calC$ is not semisimple. However, starting from $\calC$, we can obtain a fusion category as follows. Let us set
\[
 I = \left\{ 2n \Bigm\vert 0 \leqs n \leqs \frac{r-3}{2} \right\},
\]
let us consider the set of simple objects
\[
 \Gamma(\bar{\calC}) = \{ V_i \in \calC \mid i \in I \}, 
\] 
and let $\tilde{\calC}$ denote the full monoidal subcategory of $\calC$ generated by $\Gamma(\bar{\calC})$. A morphism $f \in \calC(V,W)$ is \textit{negligible} if
\[
 \tr_\calC(g \circ f) = 0
\]
for every morphism $g \in \calC(W,V)$. An example of a negligible morphism is the identity $\id_{V_{r-1}}$ of the so-called \textit{Steinberg module} $V_{r-1}$. Negligible morphisms of $\tilde{\calC}$ form a linear congruence\footnote{The term \textit{tensor ideal} is often used instead. Remark however that there are two concurring notions of tensor ideal, with the other one arising in the context of modified traces. In order to avoid confusion, we use Mac Lane's terminology here.} $\calN(\tilde{\calC})$ in $\tilde{\calC}$, see \cite[Sec.~II.8]{M71}. This means that we can define a quotient category $\bar{\calC} := \tilde{\calC} / \calN(\tilde{\calC})$ whose objects coincide with those of $\tilde{\calC}$, and whose morphism spaces are given by
\[
 \bar{\calC}(V,W) := \tilde{\calC}(V,W)/\calN(\tilde{\calC})(V,W).
\]

\begin{proposition}
 The category $\bar{\calC}$ is T-modular.
\end{proposition}

\begin{proof}
 It follows from \cite[Cor.~4.1]{A92} that $\bar{\calC}$ is semisimple. This means that $\bar{\calC}$ is a ribbon fusion category, and we only need to check that its S-matrix $S(\bar{\calC})$ is invertible. Thanks to \cite[Prop.~1.1]{B00}, this happens if and only if none of the columns of $S(\bar{\calC})$ is colinear to the one corresponding to the tensor unit $\one$. Then, let us consider the matrix
 \[
  S(\calC) := \left( \pic{S-matrix} \right)_{0 \leqs i,j \leqs r-1}.
 \]
 It follows from an easy computation, which mimics \cite[Lem.~3.27]{KM91}, that
 \[
  S(\calC)_{i,j} = [(i+1)(j+1)]
 \]
 for all $0 \leqs i,j \leqs r-1$. Remark that this formula explains why we need to restrict ourselves to the full subcategory $\tilde{\calC}$ of $\calC$. Indeed, the column $(S_{i,j})_{0 \leqs j \leqs r-1}$ is colinear to the column $(S_{r-i-2,j})_{0 \leqs j \leqs r-1}$ for every integer $0 \leqs i \leqs r-1$.  We claim now that the S-matrix $S(\bar{\calC})$ is given by the submatrix 
 \[
  S := \left( S(\calC)_{2i,2j} \right)_{0 \leqs i,j \leqs \frac{r-3}{2}}
 \]
 of $S(\calC)$. The fact that no column of $S$ is colinear to the first one follows from
 \begin{align*}
  &\sum_{i=0}^{\frac{r-3}{2}} [2i+1] S_{i,j} = \sum_{i=0}^{\frac{r-3}{2}} [2i+1][(2i+1)(2j+1)] \\
  &\hspace*{\parindent} = \frac 12 \left( \sum_{i=0}^{\frac{r-3}{2}} [2i+1][(2i+1)(2j+1)] + \sum_{i=0}^{\frac{r-3}{2}} [2i+1][(2i+1)(2j+1)] \right) \\
  &\hspace*{\parindent} = \frac 12 \left( \sum_{i=0}^{\frac{r-3}{2}} [2i+1][(2i+1)(2j+1)] + \sum_{i=0}^{\frac{r-3}{2}} [r-2i-1][(r-2i-1)(2j+1)] \right) \\
  &\hspace*{\parindent} = \frac 12 \left( \sum_{i=0}^{\frac{r-3}{2}} [2i+1][(2i+1)(2j+1)] + \sum_{i=1}^{\frac{r-1}{2}} [2i][2i(2j+1)] \right) \\
  &\hspace*{\parindent} = \frac 12 \sum_{i=0}^{r-1} [i][i(2j+1)] 
  = \frac 12 \sum_{i=0}^{r-1} \frac{\{ i \} \{ i(2j+1) \}}{\{ 1 \}^2} 
  = \frac 12 \sum_{i=0}^{r-1} \frac{\{ 2i(j+1) \}' - \{ 2ij \}'}{\{ 1 \}^2} \\ 
  &\hspace*{\parindent} = - \delta_{0,j} \frac{r}{\{ 1 \}^2}. \qedhere
 \end{align*}
\end{proof}

It follows then from the construction that $\Gamma(\bar{\calC})$ is a set of representatives of isomorphism classes of simple objects of $\bar{\calC}$. A direct computation gives the \textit{stabilization parameters}
\begin{align*}
 \Delta_- &= -\frac{i^{\frac{r-1}{2}} r^{\frac 12} q^{\frac{r+3}{2}}}{\{ 1 \}}, &
 \Delta_+ &= \frac{i^{-\frac{r-1}{2}} r^{\frac 12} q^{\frac{r-3}{2}}}{\{ 1 \}},
\end{align*}
and we fix the square root
\[
 \calD = \frac{i r^{\frac 32}}{\{ 1 \}}
\]
of the product $\Delta_+ \Delta_-$, which uniquely determines the coefficient
\[
 \delta = i^{-\frac{r+1}{2}} q^{\frac{r-3}{2}},
\]
as well as the corresponding normalization of $\RT_{\bar{\calC}}$. 

Thanks to equation \eqref{E:E:st_sp_sem-sim_dec}, state spaces of closed oriented surfaces for the TQFT $\rmV_{\bar{\calC}}$ extending $\RT_{\bar{\calC}}$ can be written as direct sums of tensor products of multiplicity modules. It turns out that all multiplicity modules in $\bar{\calC}$ have dimension at most one. In particular, if $\Sigma_g$ is a closed oriented surface of genus $g$, then a basis of $\rmV_{\bar{\calC}}(\Sigma_g)$ is determined by so-called $r$-triangular labelings of the set of edges $\calE_g$ of the graph $\calG_g$ represented in \eqref{E:triv_graph}. If $v \in \calV_g$ is a vertex, and $v_- \cup v_+ = \{ e,e',e'' \}$ is the set of germs of edges which are incident to $v$, then a labeling $i : \calE_g \rightarrow I$ is said to be \textit{$r$-triangular in $v$} if it satisfies the inequalities
\begin{equation*}
 \lvert i(e')-i(e'') \rvert \leqs i(e) \leqs i(e')+i(e''), \qquad i(e)+i(e')+i(e'') < 2r-2.
\end{equation*}
A labeling $i : \calE_g \rightarrow I$ is \textit{$r$-triangular}\footnote{Such a labeling is called \textit{$r$-admissible} in \cite{BHMV95}, but we change terminology here because admissibility will be a different condition, and a key one, in the non-semisimple construction.} if it is $r$-triangular in every $v \in \calV_g$. Thanks to the next statement, the dimension of $\rmV_{\bar{\calC}}(\Sigma_g)$ is obtained by counting $r$-triangular labelings of $i : \calE_g \rightarrow I$.

\begin{proposition}
 The multiplicity module
 \[
  \bar{\calC}(i,v) \cong \bar{\calC}(\one,V_{i(e)} \otimes V_{i(e')} \otimes V_{i(e'')})
 \]  
 has dimension one if $i$ is $r$-triangular in $v$, and zero otherwise.
\end{proposition}

\begin{proof}
 Since $\bar{\calC}$ is a fusion category, the dimension of $\bar{\calC}(\one,V_{i(e)} \otimes V_{i(e')} \otimes V_{i(e'')})$ is equal to the multiplicity of $V_{i(e)}^* \cong V_{i(e)}$ in the decomposition of $V_{i(e')} \otimes V_{i(e'')}$ in $\bar{\calC}$. The decomposition of $V_{i(e')} \otimes V_{i(e'')}$ into a sum of indecomposable $\bar{U}_q \fsl_2$-modules in $\calC$ is given by
 \[
  V_{i(e')} \otimes V_{i(e'')} \cong \bigoplus_{n = \frac{\lvert i(e')-i(e'') \rvert}{2}}^{\min \left\{ \frac{i(e')+i(e'')}{2},r-\frac{i(e')+i(e'')}{2}-2 \right\}} V_{2n} \oplus \bigoplus_{n = r-\frac{i(e')+i(e'')}{2}-1}^{\frac{r-1}{2}} P_{2n},
 \]
 where $P_{2n}$ is the projective cover of $V_{2n}$, see Section \ref{S:sl_2_non-sem-sim}. A similar formula for even values of $r$ is established in \cite{S92}, and the proof can be adapted. Since $P_{2n}$ is a direct summand of $V_{r-1} \otimes V_{r-1}$, its identity is negligible. In particular, the decomposition of $V_{i(e')} \otimes V_{i(e'')}$ in the category $\bar{\calC}$ reduces to
 \[
  V_{i(e')} \otimes V_{i(e'')} \cong \bigoplus_{n = \frac{\lvert i(e')-i(e'') \rvert}{2}}^{\min \left\{ \frac{i(e')+i(e'')}{2},r-\frac{i(e')+i(e'')}{2}-2 \right\}} V_{2n}.
 \]
 This proves the proposition.
\end{proof}

As an example, the state space of the torus $\Sigma_1 = S^1 \times S^1$ has dimension
\[
 \dim_\C \rmV_\calC(\Sigma_1) = \frac{r-1}{2}.
\]
Since every Dehn twist is sent to (the projective class of) a matrix of finite order, the projective representation of the genus one mapping class group $\MCG(\Sigma_1)$ in $\PGL_\C(\rmV_\calC(\Sigma_1))$ is never faithful, although it has very interesting asymptotic faithfulness properties as $r$ tends to infinity, see \cite[Thm.~1]{A02}.

\section{Lyubashenko's notion of modular category}\label{S:L-mod_cat}

An \textit{L-modular category} (L for Lyubashenko) can be defined as a finite ribbon category whose Müger center is trivial. Remark that this is not the original definition of Lyubashenko, but it is equivalent to it. This is a consequence of a result of Shimizu \cite{S16}, who compared several notions of non-degeneracy for the braiding of a finite ribbon category, proving the equivalence of four different conditions. Among these, we can find Lyubashenko's original one, as well as the short property used above. This formulation makes it clear that T-modular categories are just semisimple L-modular categories.

In order to review Lyubashenko's definition, let $\calC$ be a finite ribbon category, and let us consider the functor $\calH : \calC^\op \times \calC \to \calC$ determined by $(X,Y) \mapsto X^* \otimes Y$. We recall the classical notion of a \textit{dinatural transformation of source $\calH$}, which is given by an object $Z \in \calC$ equipped with a family $\{ d_X : X^* \otimes X \to Z \}_{X \in \calC}$ of morphisms of $\calC$ satisfying
\[
 d_X \circ (f^* \otimes \id_X) = d_Y \circ (\id_{Y*} \otimes f)
\]
for every morphism $f : X \to Y$ of $\calC$, see \cite[Sec.~IX.4]{M71}. We usually abuse notation, and simply denote dinatural transformations by their underlying objects. The \textit{coend of $\calH$} is the universal dinatural transformation of source $\calH$, which means it is an object $\calL \in \calC$ with a dinatural family of structure morphisms $i_X : X^* \otimes X \to \calL$ satisfying the following condition: for every dinatural transformation $Z$ of source $\calH$ there exists a unique morphism $\ell_Z : \calL \to Z$ such that the diagram
\begin{equation}\label{E:univ_prop_coend}
 \begin{tikzpicture}[descr/.style={fill=white},baseline=(current  bounding  box.center)]
  \node (P0) at (135:{sqrt(2)*2.5}) {$Y^* \otimes X$};
  \node (P1) at (90:2.5) {$X^* \otimes X$};
  \node (P2) at (180:2.5) {$Y^* \otimes Y$};
  \node (P3) at (0:0) {$\calL$};
  \node (P4) at (-45:{2.5/sqrt(2)}) {$Z$};
  \draw
  (P0) edge[->] node[above] {\scriptsize $f^* \otimes \id_X$} (P1)
  (P1) edge[->] node[right] {\scriptsize $i_X$} (P3)
  (P0) edge[->] node[left] {\scriptsize $\id_{Y^*} \otimes f$} (P2)
  (P2) edge[->] node[below] {\scriptsize $i_Y$} (P3)
  (P1) edge[bend left,->] node[right] {\scriptsize $d_X$} (P4)
  (P2) edge[bend right,->] node[below] {\scriptsize $d_Y$} (P4)
  (P3) edge[->] node[descr] {\scriptsize $\ell_Z$} (P4);
 \end{tikzpicture}
\end{equation}
commutes for every morphism $f : X \to Y$ of $\calC$. It has been proved by Majid \cite{M91a} and Lyubashenko \cite{L95,L99} that, since $\calC$ is a finite ribbon category, the coend of $\calH$ always exists, and furthermore it supports the structure of a braided Hopf algebra in $\calC$. This means that it admits:
\begin{enumerate}
 \item a product $\mu : \calL \otimes \calL \to \calL$ and a unit $\eta : \one \to \calL$;
 \item a coproduct $\Delta : \calL \to \calL \otimes \calL$ and a counit $\epsilon : \calL \to \one$;
 \item an antipode $S : \calL \to \calL$.
\end{enumerate}
All these structure morphisms are uniquely determined by the universal property~\eqref{E:univ_prop_coend} of $\calL$. Indeed, they can be defined as the unique morphisms of $\calC$ that satisfy
\begin{gather*}
 \pic{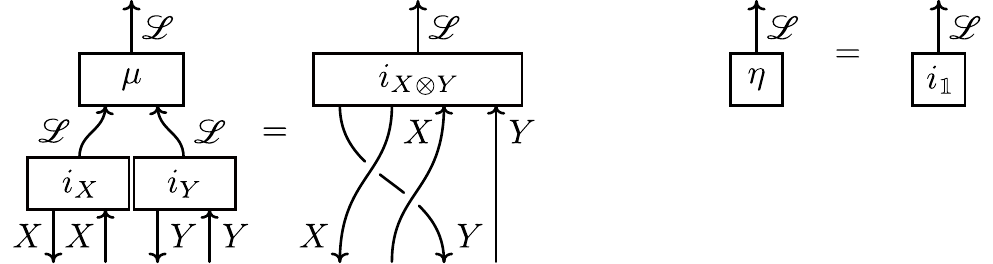} \\[5pt]
 \pic{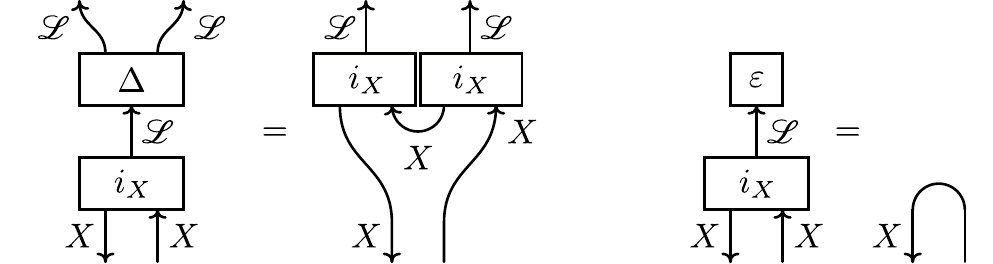} \\[5pt]
 \pic{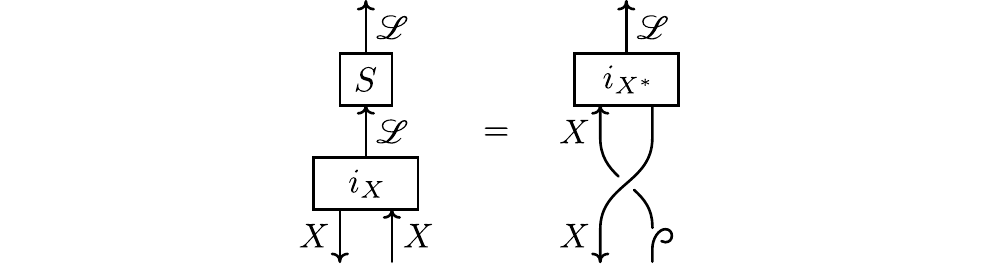}
\end{gather*}
for all $X,Y \in \calC$. Remark that the product $\mu$ is actually determined by a universal property satisfied by the so-called \textit{double coend} $\calL \otimes \calL$, and it requires a Fubini theorem for coends \cite[Prop.~IX.8]{M71} whose name comes from the standard notation
\[
 \calL = \int^{X \in \calC} X^* \otimes X,
\]
see also \cite[Prop.~5.1.2]{KL01}. We should point out that, dually, every finite ribbon category $\calC$ also admits an \textit{end}
\[
 \calE = \int_{X \in \calC} X \otimes X^*,
\]
which is similarly defined as a universal dinatural transformation, see Appendix \ref{A:end} for details. Equivalent constructions can be carried out exploiting the properties of $\calE$, which are completely analogous to those of $\calL$, and it is for historical reasons that we focus on $\calL$ instead of $\calE$, following Lyubashenko.

From a topological point of view, what we said so far hints at the fact that the universal property \eqref{E:univ_prop_coend} of the coend $\calL$ provides an alternative to the Reshetikhin--Turaev functor for extracting morphisms of $\calC$ from diagrams of certain framed oriented tangles. For instance, a Hopf pairing $\omega : \calL \otimes \calL \to \one$ is uniquely determined by the equality
\[
 \pic{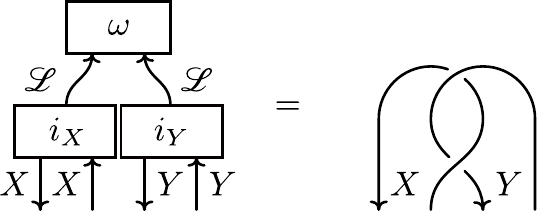}
\]
for all $X,Y \in \calC$. Then, following Lyubashenko's original definition, an L-modular category is a finite ribbon category for which the Hopf pairing $\omega$ is non-degenerate, meaning that it induces an isomorphism between $\calL$ and $\calL^*$. Remark how similar the defining diagram for $\omega$ is with respect to the one defining the S-matrix of Turaev. In fact, when $\calC$ is semisimple, then Prop.~\ref{P:end_fusion} and \cite[Lem.~2]{K96a} give
\[
 \calE \cong \bigoplus_{i \in I} V_i \otimes V_i^*, \qquad 
 \calL \cong \bigoplus_{i \in I} V_i^* \otimes V_i,
\]
and $\omega$ is non-degenerate if and only if the S-matrix is invertible \cite[Prop.~7.4.3]{KL01}.

An alternative condition for the modularity of $\calC$ is that a certain morphism $D : \calL \to \calE$ called the \textit{Drinfeld map} is an isomorphism, see \cite[Prop.~4.11]{FGR17}. When $\calC$ is the category $\mods{H}$ of finite-dimensional representations of a finite-dimensional ribbon Hopf algebra $H$, then let us denote with $\ad$ the \textit{adjoint representation} of $H$, which is given by the vector space $H$ equipped with the adjoint left $H$-action defined by
\[
 \ad_x(y) = x_{(1)}yS(x_{(2)})
\]
for all $x,y \in H$, where we are using Sweedler's notation $\Delta(x) = x_{(1)} \otimes x_{(2)}$ for the coproduct. Similarly, let us denote with $\coad$ the \textit{coadjoint representation} of $H$, which is given by the vector space $H^*$ equipped with the coadjoint left $H$-action defined by
\[
 \coad_x(\phi)(y) = \phi(S(x_{(1)})yx_{(2)})
\]
for all $x,y \in H$ and $\phi \in H^*$. Then Prop.~\ref{P:end_H-mod} and \cite[Lem.~3]{K96a} give
\[
 \calE \cong \ad, \qquad \calL \cong \coad.
\]
In this setting, the Drinfeld map $D : \coad \to \ad$ was first defined in \cite[Prop.~3.3]{D90} as the $H$-module morphism determined by $D(\phi) := (\phi \otimes \id_H)(R_{21}R_{12})$, where $R_{12} = R \in H \otimes H$ is the R-matrix of $H$, and $R_{21} \in H \otimes H$ is obtained from $R$ by reversing the order of its components. The element $M = R_{21}R_{12} \in H \otimes H$ is sometimes called the \textit{M-matrix}, or \textit{monodromy matrix}, of $H$. By definition, $H$ is \textit{factorizable} if $D$ is an isomorphism.

When $\calC$ is modular, an \textit{integral} of the coend $\calL$ is a morphism\footnote{The source of $\Lambda$ might be a different invertible object in general, called the \textit{object of integrals}, but here we are supposing $\calC$ is modular.} $\Lambda : \one \to \calL$ satisfying
\[
 \mu \circ (\Lambda \otimes \id_\calL) = \Lambda \circ \epsilon = \mu \circ (\id_\calL \otimes \Lambda).
\]
In analogy with Sweedler's result for Hopf algebras in the category of finite-di\-men\-sion\-al vector spaces, the coend $\calL$ admits an integral, which is unique up to scalar \cite[Thm.~6.1]{L95}. When $\calC$ is semisimple, then, up to scalar, we have
\[
 \Lambda = \left( \dim_\calC(V_i) \cdot \rcoev_{V_i} \right)_{i \in I},
\]
which means that we can think of an integral $\Lambda$ of the coend $\calL$ as extending the Kirby color to the non-semisimple case. When $\calC$ is the category $\mods{H}$ of finite-dimensional representations of a finite-dimensional factorizable ribbon Hopf algebra $H$, then, up to scalar, we have
\[
 \Lambda(1) = \lambda \in H^*,
\]
where $\lambda \in H^*$ is a classical right integral on the Hopf algebra $H$. See \cite[Sec.~2.5]{K96a} for both claims.

The algebraic setup discussed up to here is sufficient to perform Lyubashenko's construction, but in order to obtain TQFTs, we also need the theory of modified traces. Let us present the key definitions by focusing on a relevant special case. We denote by $\Proj(\calC)$ the full subcategory of projective objects of $\calC$, which forms a \textit{tensor ideal} in $\calC$, in the sense that it is closed under retracts and absorbent under tensor products, see \cite[Def.~3.1.1]{GKP10}. A \textit{trace} on $\Proj(\calC)$ is a family of linear maps
\[
 \rmt := \{ \rmt_P : \End_\calC(P) \to \Bbbk \}_{P \in \Proj(\calC)}
\]
satisfying:
\begin{enumerate}
 \item \textit{Cyclicity}: $\rmt_P \left( \pic{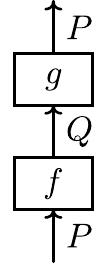} \right) = \rmt_Q \left( \pic{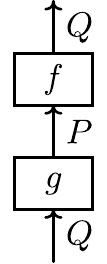} \right)$ \par \bigskip
  for all $P,Q \in \Proj(\calC)$, $f \in \calC(P,Q)$, and $g \in \calC(Q,P)$; 
  \bigskip
 \item \textit{Partial trace}: $\rmt_{P \otimes X} \left( \pic{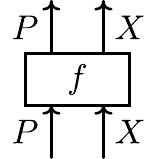} \right) = \rmt_P \left( \pic{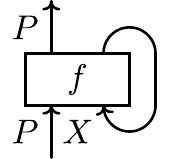} \right)$ \par \bigskip
  for all $P \in \Proj(\calC)$, $X \in \calC$, and $f \in \End_{\calC}(P \otimes X)$.
\end{enumerate}
We say a trace $\rmt$ on $\Proj(\calC)$ is \textit{non-degenerate} if, for all $P \in \Proj(\calC)$ and $X \in \calC$, the bilinear pairing $\rmt_P(\_ \circ \_) : \calC(X,P) \times \calC(P,X) \to \Bbbk$ determined by $(g,f) \mapsto \rmt_P(g \circ f)$ is non-degenerate. Thanks to \cite[Cor.~5.6]{GKP18}, for a modular\footnote{The statement is a particular case of a much more general existence and uniqueness result, see \cite[Thm.~5.5]{GKP18}.} category $\calC$ there exists a trace $\rmt$ on $\Proj(\calC)$, which is unique up to scalar and furthermore non-degenerate. When $\calC$ is semisimple, then $\Proj(\calC) = \calC$, and the standard categorical trace $\tr_\calC$ is the unique trace on $\calC$ up to scalar. When $\calC$ is the category $\mods{H}$ of finite-dimensional representations of a finite-dimensional factorizable\footnote{Again, hypotheses for this result can be considerably relaxed, but we prefer to avoid introducing more terminology.} ribbon Hopf algebra $H$, there is a natural correspondence between the space of traces on $\Proj(\calC)$ and the spaces of so-called \textit{symmetrized} integrals on $H$, see \cite[Thm.~1]{BBG18}.

\section{Non-semisimple quantum invariants and TQFTs}\label{S:TQFT_non-sem-sim}

Using the coend $\calL$ and the integral $\Lambda$, it is possible to develop a diagrammatic calculus which steps outside of the framework provided by the Reshetikhin--Turaev functor. This time, the topological gadgets underlying this graphical notation are called \textit{bichrome graphs}. Very roughly speaking, bichrome graphs are ribbon graphs whose edges can be of two kinds, either \textit{red} (and unlabeled) or \textit{blue} (and labeled as usual by objects of $\calC$). Similarly, coupons (which, for standard ribbon graphs, correspond to boxes like the one represented in \eqref{E:coupon}) are partitioned in two groups, according to the edges they meet, and can thus be either \textit{bichrome} (and unlabeled) or \textit{blue} (and labeled as usual by morphisms of $\calC$). These ribbon graphs are subject to certain restrictions, see \cite[Sec.~3.1]{DGGPR19} for details. An example of a bichrome graph, with red components represented as dashed-dotted lines, is given by
\[
 \pic{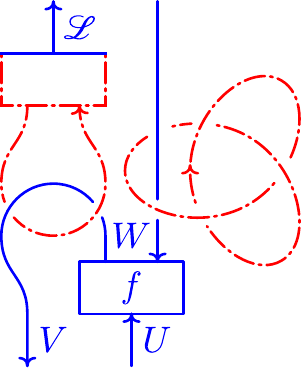}
\]
This can be understood as representing a morphism of $\calC$ with source $V^* \otimes U$ and target $\calL \otimes W^*$. Indeed, just like standard ribbon graphs, bichrome graphs can be arranged as morphisms of a category\footnote{Our notation for the category of bichrome graphs $\calR_\Lambda$ features the integral $\Lambda$ in order to highlight the relevant functor this category supports, but remark that the definition of $\calR_\Lambda$ itself is actually independent of the choice of $\Lambda$.} $\calR_\Lambda$, and we can interpret $\calR_\calC$ as the subcategory whose morphisms are exclusively blue bichrome graphs. Then, the idea behind bichrome graphs is essentially the following: on the one hand, the red part of the graph should be interpreted as prescribing surgery instructions, and it has to be processed through Lyubashenko's algorithm, that is, invoking the universal property \eqref{E:univ_prop_coend} of the coend $\calL$, in order to extract a morphism of $\calC$ (red minima of diagrams translate to the integral $\Lambda$ under this correspondence); on the other hand, the blue part supports the standard graphical calculus implemented by the Reshetikhin--Turaev functor. There exists a coherent way of harmonizing the two approaches, which is realized by a functor
\[
 F_\Lambda : \calR_\Lambda \to \calC
\]
restricting to $F_\calC$ on $\calR_\calC$, and called the \textit{Lyubashenko--Reshetikhin--Turaev functor}, see \cite[Sec.~3.1]{DGGPR19}.

Starting from an L-modular category $\calC$, Lyubashenko's construction can be reformulated as producing a topological invariant $\rmL_\calC$ of closed oriented 3-manifolds decorated with closed bichrome graphs. As explained in Section \ref{SS:renormalized_non-semisimple}, this invariant cannot be extended to a TQFT unless $\calC$ is semisimple, in which case everything boils down to Turaev's theory. However, in the non-semisimple case, one can construct a renormalized invariant $\rmL'_\calC$ of closed oriented 3-manifolds decorated with so-called admissible closed bichrome graphs. We say a bichrome graph is \textit{admissible} if it features at least a projective object $P \in \Proj(\calC)$ among the labels of its blue edges. In particular, just like in the semisimple case, the theory can be understood as producing a family $\{ \rmL'_P \}_{P \in \Proj(\calC)}$ of topological invariants of framed oriented links inside 3-manifolds. However, this time it is crucial not to allow empty links, and to restrict to projective objects for the parametrization.

The first step for the construction of $\rmL'_\calC$ is the definition of a renormalized invariant of admissible closed bichrome graphs. A \textit{cutting presentation} of an admissible closed bichrome graph $T$ is a bichrome graph $T_P$ satisfying
\[
 \pic{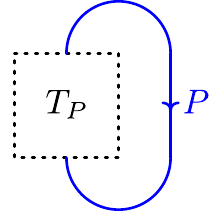} = T
\]
for some $P  \in \Proj(\calC)$. Remark that every admissible bichrome graph admits a cutting presentation, which is simply obtained by cutting open one of its blue edges carrying projective label. However, cutting presentations are by no means unique. For instance, we might have several blue edges with projective labels to choose from, and these could all determine inequivalent presentations. Nonetheless, if $T$ is an admissible closed bichrome graph, and if $T_P$ is a cutting presentation of $T$, then 
\[
 F'_\calC(T) := \rmt_P \left( F_\Lambda(T_P) \right)
\]
is a topological invariant of $T$. The proof of this result essentially follows from the defining properties of the trace $\rmt$, which were designed explicitly to obtain these kinds of \textit{renormalized} topological invariants. Now, in order to extend this construction to 3-manifolds, we need again a normalization coefficient, just like in the semisimple theory. This is defined in terms of the so-called \textit{stabilization coefficients}
\[
 \Delta_+ := F_\Lambda \left( \pic{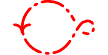} \right) \in \Bbbk, \qquad
 \Delta_- := F_\Lambda \left( \pic{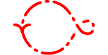} \right) \in \Bbbk.
\]
Just like in the fusion case, it is a consequence of L-modularity that the product $\Delta_+ \Delta_-$ is invertible, and so we can choose two parameters $\calD,\delta \in \Bbbk^\times$ satisfying
\[
 \calD^2 = \Delta_+ \Delta_-, \qquad 
 \delta = \frac{\Delta_+}{\calD} = \frac{\calD}{\Delta_-}.
\]
The definition of the topological invariant $\rmL'_\calC$ depends again on the choice of $\calD$, which uniquely determines $\delta$. If $M$ is a closed oriented 3-manifold, if $T \subset M$ is an admissible closed bichrome graph, and if $L \subset S^3$ is a surgery presentation of $M$, which we assume having $\ell$ components and signature $\sigma$, and to which we assign an arbitrary orientation and the color red, then
\[
 \rmL'_\calC(M,T) := \calD^{-1-\ell} \delta^{-\sigma} F'_\Lambda(L \cup T)
\]
is a topological invariant of the pair $(M,T)$.

Just like in Section \ref{S:TQFT_sem-sim}, the universal construction of \cite{BHMV95} extends the invariant $\rmL'_\calC$ to a dual pair of TQFTs given by a covariant functor $\rmV_\calC$ and a contravariant one $\rmV'_\calC$.  Again, this means that, for every closed oriented surface $\Sigma$, we obtain dual vector spaces $\rmV_\calC(\Sigma)$ and $\rmV'_\calC(\Sigma)$ supporting respectively left and right actions of decorated cobordisms. This time, if $\bar{\Sigma}$ is obtained from $\Sigma$ by reversing its orientation, the isomorphism $\rmV'_\calC(\Sigma) \cong \rmV_\calC(\bar{\Sigma})$ is no longer natural, as a consequence of the admissibility condition. Indeed, if $\Sigma_g$ is a closed oriented surface of genus $g$, then the contravariant state space $\rmV'_\calC(\Sigma_g)$ can be naturally identified with
\begin{equation}\label{E:st_sp_non-sem-sim_prime}
 \rmV_\calC'(\Sigma_g) \cong \calC(\calL^{\otimes g},\one).
\end{equation}
On the other hand, the covariant state space $\rmV_\calC(\Sigma_g)$, which is isomorphic to the linear dual of $\rmV'_\calC(\Sigma_g)$, can be naturally identified with 
\begin{equation}\label{E:st_sp_non-sem-sim}
 \rmV_\calC(\Sigma_g) \cong \calC(P_{\one},\calE^{\otimes g}) / \rad_R \langle \_,\_ \rangle_g,
\end{equation}
where $P_{\one}$ is the projective cover of the tensor unit $\one$, and where the bilinear pairing $\langle \_,\_ \rangle_g : \calC(\calL^{\otimes g},\one) \times \calC(P_{\one},\calE^{\otimes g}) \to \Bbbk$ is determined by 
\[
 \langle f',f \rangle_g = \rmt_{P_{\one}} \left( \eta_{\one} \circ f' \circ (D^{-1})^{\otimes g} \circ f \right)
\]
for the Drinfeld map $D : \calL \to \calE$ and the injective envelope\footnote{Projective covers coincide with injective envelopes in modular categories.} monomorphism $\eta_{\one} : \one \to P_{\one}$. When $\calC$ is semisimple, then $P_{\one} = \one$, $\eta_{\one} = \id_{\one}$, and $\langle \_,\_ \rangle_g$ is non-degenerate, so that we recover the description of \eqref{E:st_sp_sem-sim}.

\section{L-modular categories from small quantum \texorpdfstring{$\mathfrak{sl}_2$}{sl(2)}}\label{S:sl_2_non-sem-sim}

The small quantum group $\bar{U}_q \fsl_2$ introduced in Section \ref{S:sl_2_sem-sim} is factorizable. This is proved in \cite[Ex.~3.4.3]{M95} by looking directly at the M-matrix $M \in \bar{U}_q \fsl_2 \otimes \bar{U}_q \fsl_2$, which is given by
\[
 M := \frac{1}{r} \sum_{a,b,c,d=0}^{r-1} \frac{\{ 1 \}^{a+b}}{[a]![b]!} q^{\frac{a(a-1)+b(b-1)}{2} - 2 cd - (b+c)(b-d)} F^b K^c E^a \otimes E^b K^d F^a.
\]
This means the category $\calC = \mods{\bar{U}_q \fsl_2}$ of finite-dimensional left $\bar{U}_q \fsl_2$-modules is L-modular, see \cite[Rem.~3.6.1]{L95}.

An \textit{integral} $\Lambda : \C \to \coad$ is determined by the linear form $\lambda = \Lambda(1) \in (\bar{U}_q \fsl_2)^*$ defined as
\[
 \lambda \left( E^a F^b K^c \right) := \frac{r^3}{\{ 1 \}^{2r-2}} \delta_{a,r-1} \delta_{b,r-1} \delta_{c,1},
\]
see \cite[Prop.~A.5.1]{L94}. Remark that Lyubashenko uses Lusztig's coproduct, which explains why the formula written here defines a right integral, instead of a left one. A direct computation gives the \textit{stabilization parameters}
\begin{align*}
 \Delta_- &= \lambda(v) = i^{\frac{r-1}{2}} r^{\frac 32} q^{\frac{r+3}{2}}, &
 \Delta_+ &= \lambda(v^{-1}) = i^{-\frac{r-1}{2}} r^{\frac 32} q^{\frac{r-3}{2}},
\end{align*}
and we fix the square root
\[
 \calD = r^{\frac 32}
\]
of the product $\Delta_+ \Delta_-$, which uniquely determines the coefficient
\[
 \delta = i^{-\frac{r-1}{2}} q^{\frac{r-3}{2}},
\]
as well as the corresponding normalization of $\rmL'_\calC$. 

In order to illustrate the difference between the L-modular category $\calC$ and its T-modular subquotient $\bar{\calC}$ defined in Section \ref{S:sl_2_sem-sim}, let us describe projective covers of simple objects of $\calC$. For every integer $0 \leqs n \leqs r-2$ we denote with $P_n$ the $\bar{U}_q \fsl_2$-module with basis 
\[
 \{ a_i^n, x_j^n, y_j^n, b_i^n \in P_n \mid 0 \leqs i \leqs n, 0 \leqs j \leqs r-n-2 \}
\] 
and action given, for all integers $0 \leqs i \leqs n$ and $0 \leqs j \leqs r-n-2$, by
\begin{equation*}
 \begin{split}
  K \cdot a_i^n &= q^{n-2i} a_i^n, \\
  E \cdot a_i^n &= [i][n-i+1] a_{i-1}^n, \\
  F \cdot a_i^n &= a_{i+1}^n, \\
  K \cdot x_j^n &= q^{-n-2j-2} x_j^n, \\
  E \cdot x_j^n &= -[j][n+j+1] x_{j-1}^n, \\ 
  F \cdot x_j^n &= 
  \begin{cases}
   x_{j+1}^n & 0 \leqs j < r-n-2, \\
   a_0^n & j = r-n-2,
  \end{cases} \\
  K \cdot y_j^n &= q^{-n-2j-2} y_j^n, \\
  E \cdot y_j^n &= 
  \begin{cases}
   a_n^n & j = 0, \\
   -[j][n+j+1] y_{j-1}^n & 0 < j \leqs r-n-2,
  \end{cases} \\
  F \cdot y_j^n &= y_{j+1}^n, \\
  K \cdot b_i^n &= q^{n-2i} b_i^n, \\
  E \cdot b_i^n &= 
  \begin{cases}
   x_{r-n-2}^n & i = 0, \\
   a_{i-1}^n + [i][n-i+1] b_{i-1}^n & 0 < i \leqs n, 
  \end{cases} \\
  F \cdot b_i^n &= 
  \begin{cases}
   b_{i+1}^n & 0 \leqs i < n, \\
   y_0^n & i = n,
  \end{cases}
 \end{split}
\end{equation*}
where $a_{-1}^n := a_{n+1}^n := x_{-1}^n := y_{r-n-1}^n := 0$. These modules are all projective, indecomposable, and self-dual, and every indecomposable projective object of $\calC$ is isomorphic to either $P_{r-1} = V_{r-1}$ or $P_n$ for some $0 \leqs n \leqs r-2$. Their categorical dimension is
\[
 \dim_\calC(P_n) = \tr_\calC(\id_{P_n}) = 0.
\]
However, the normalization $\rmt_{V_{r-1}}(\id_{V_{r-1}}) = 1$ uniquely determines a trace $\rmt$ on $\Proj(\calC)$ which satisfies
\[
 \rmt_{P_n}(\id_{P_n}) = \{ n+1 \}',
\]
as follows from a computation that can be adapted from \cite[Lem.~6.9]{CGP14} (remark the difference between the pivotal element $g = K$ and the one used by Costantino, Geer, and Patureau for their right duality morphisms, which accounts for the sign appearing in their formula).

Thanks to equations \eqref{E:st_sp_non-sem-sim_prime} and \eqref{E:st_sp_non-sem-sim}, state spaces of closed oriented surfaces for the TQFTs $\rmV_{\calC}$ and $\rmV'_{\calC}$ extending $\rmL'_{\calC}$ can be written in terms of morphism spaces in $\calC$ involving tensor powers of the end $\calE \cong \ad$ and of the coend $\calL \cong \coad$. We point out that, although $\ad$ and $\coad$ can be described very explicitly, their description is rather complicated. For instance, they are not projective objects of $\calC$, as follows from an explicit decomposition into indecomposable summands computed by Ostrik \cite{Os95}. If $\Sigma_g$ is a closed oriented surface of genus $g$, then equation \eqref{E:st_sp_non-sem-sim_prime} identifies $\rmV'_{\calC}(\Sigma_g)$ with $ \calC(\coad^{\otimes g},\one)$. On the other hand, equation \eqref{E:st_sp_non-sem-sim} describes the dual vector space $\rmV_{\calC}(\Sigma_g)$ as a quotient of $\calC(P_0,\ad^{\otimes g})$. In the example of the torus $\Sigma_1 = S^1 \times S^1$, since $\calC(\coad,\one)$ can be naturally identified with the center $Z(\bar{U}_q \fsl_2)$ of $\bar{U}_q \fsl_2$, then
\[
 \rmV'_{\calC}(\Sigma_1) \cong Z(\bar{U}_q \fsl_2).
\]
On the other hand, the quotient of $\calC(P_0,\ad^{\otimes g})$ can be naturally identified with the Hochschild homology $\HH_0(\bar{U}_q \fsl_2) = \bar{U}_q \fsl_2/[\bar{U}_q \fsl_2,\bar{U}_q \fsl_2]$ of $\bar{U}_q \fsl_2$, so that
\[
 \rmV_{\calC}(\Sigma_1) \cong \HH_0(\bar{U}_q \fsl_2).
\]
The dimension of both spaces is
\[
 \dim_\C \rmV_\calC(\Sigma_1) = \dim_\C Z(\bar{U}_q \fsl_2) = \frac{3r-1}{2},
\]
as computed by Kerler in \cite[Lem.~14]{K94}. In the same paper, Kerler independently redefines and computes Lyubashenko's projective representation of the genus one mapping class group $\MCG(\Sigma_1)$ in $\PGL_\C(Z(\bar{U}_q \fsl_2))$, see \cite[Thm.~2]{K94}. He conjectures an explicit decomposition of this representation, and he verifies his conjecture for $r=3$ and $r=5$ in \cite[Sec.~3.6]{K94}. Remark that in both cases the resulting representation of $\MCG(\Sigma_1)$ is faithful modulo the center $Z(\MCG(\Sigma_1)) \cong \Z/2\Z$. Up to isomorphism, Kerler--Lyubashenko's representations are precisely those induced by the contravairant TQFT $\rmV'_\calC$ for $\calC = \mods{\bar{U}_q \fsl_2}$, see \cite[Thm.~1.2]{DGGPR20}.

\appendix

\section{Ends}\label{A:end}

Let us consider the functor $\calK : \calC \times \calC^\op \to \calC$ determined by $(X,Y) \mapsto X \otimes Y^*$. A \textit{dinatural transformation of target $\calK$} is given by an object $U \in \calC$ equipped with a family $\{ d_X : U \to X \otimes X^* \}_{X \in \calC}$ of morphisms of $\calC$ satisfying
\[
 (f \otimes \id_{X^*}) \circ d_X  = (\id_Y \otimes f^*) \circ d_Y
\]
for every morphism $f : X \to Y$ of $\calC$, see \cite[Sec.~IX.4]{M71}. The \textit{end of $\calK$} is the universal dinatural transformation of target $\calK$, which means it is an object $\calE \in \calC$ with a dinatural family of structure morphisms $j_X : \calE \to X \otimes X^*$ satisfying the following condition: for every dinatural transformation $U$ of target $\calK$ there exists a unique morphism $e_U : U \to \calE$ such that the diagram
\begin{equation}\label{E:univ_prop_end}
 \begin{tikzpicture}[descr/.style={fill=white},baseline=(current  bounding  box.center)]
  \node (P0) at (-45:{sqrt(2)*2.5}) {$Y \otimes X^*$};
  \node (P1) at (0:2.5) {$X \otimes X^*$};
  \node (P2) at (-90:2.5) {$Y \otimes Y^*$};
  \node (P3) at (0:0) {$\calE$};
  \node (P4) at (135:{2.5/sqrt(2)}) {$U$};
  \draw
  (P1) edge[->] node[right] {\scriptsize $f \otimes \id_{X^*}$} (P0)
  (P3) edge[->] node[above] {\scriptsize $j_X$} (P1)
  (P2) edge[->] node[below] {\scriptsize $\id_Y \otimes f^*$} (P0)
  (P3) edge[->] node[left] {\scriptsize $j_Y$} (P2)
  (P4) edge[bend left,->] node[above] {\scriptsize $d_X$} (P1)
  (P4) edge[bend right,->] node[left] {\scriptsize $d_Y$} (P2)
  (P4) edge[->] node[descr] {\scriptsize $e_U$} (P3);
 \end{tikzpicture}
\end{equation}
commutes for every morphism $f : X \to Y$ of $\calC$.

\begin{proposition}\label{P:end_fusion}
 If $\calC$ is a rigid fusion category and $\Gamma(\calC) = \left\{ V_i \mid i \in I \right\}$ is a set of representatives of isomorphism classes of simple objects of $\calC$, then
 \[
  \calE \cong \bigoplus_{i \in I} V_i \otimes V_i^*,
 \]
 and for every $X \in \calC$ the structure morphism $j_X : \calE \to X \otimes X^*$ is given by
 \[
  j_X = \left( \sum_{a=1}^{n(X,i)} \iota_a \otimes (\pi^a)^* \right)_{i \in I}
 \]
 where $\{ \iota_a \in \calC(V_i,X) \mid 1 \leqs a \leqs n(X,i) \}$ and $\{ \pi^a \in \calC(X,V_i) \mid 1 \leqs a \leqs n(X,i) \}$ are dual bases with respect to composition. Furthermore, if $U \in \calC$ is a dinatural transformation of target $\calK$ with structure morphisms $d_X : U \to X \otimes X^*$, then the unique morphism $e_U : U \to \calE$ making \eqref{E:univ_prop_end} into a commutative diagram is given by
 \[
  e_U := \left( d_{V_i} \right)_{i \in I}.
 \]
\end{proposition}

\begin{proof}
 The fact that $j_X$ is well-defined is a direct consequence of a standard linear algebra argument. Indeed, for any other pair of dual bases
 \[
  \{ \kappa_a \in \calC(V_i,X) \mid 1 \leqs a \leqs n(X,i) \}, \qquad
  \{ \rho^a \in \calC(X,V_i) \mid 1 \leqs a \leqs n(X,i) \}
 \]
 there exist scalars $m^a_b \in \Bbbk$ for $1 \leqs a,b \leqs n(X,i)$ satisfying
 \[
  \kappa_b = \sum_{a=1}^{n(X,i)} m^a_b \iota_a, \qquad
  \pi^a = \sum_{b=1}^{n(X,i)} m^a_b \rho^b,
 \]
 which means
 \[
  \sum_{b=1}^{n(X,i)} \kappa_b \otimes (\rho^b)^* = \sum_{a,b=1}^{n(X,i)} m^a_b \iota_a \otimes (\rho^b)^* = \sum_{a=1}^{n(X,i)} \iota_a \otimes (\pi^a)^*.
 \]

 Next, we claim that
 \[
  (f \otimes \id_{X^*}) \circ j_X = (\id_Y \otimes f^*) \circ j_Y
 \]
 for every morphism $f : X \to Y$. Indeed, if
 \begin{align*}
  &\{ \iota_a \in \calC(V_i,X) \mid 1 \leqs a \leqs n(X,i) \}, &
  &\{ \pi^a \in \calC(X,V_i) \mid 1 \leqs a \leqs n(X,i) \}, \\*
  &\{ \kappa_a \in \calC(V_i,Y) \mid 1 \leqs a \leqs n(Y,i) \}, &
  &\{ \rho^a \in \calC(Y,V_i) \mid 1 \leqs a \leqs n(Y,i) \}
 \end{align*}
 are dual bases for every for $i \in I$, there exist scalars $(f_i)^a_b \in \Bbbk$ for all $1 \leqs a \leqs n(Y,i)$ and $1 \leqs b \leqs n(X,i)$ satisfying
 \[
  f = \sum_{i \in I} f_i, \qquad
  f_i = \sum_{a=1}^{n(Y,i)} \sum_{b=1}^{n(X,i)} (f_i)^a_b \kappa_a \circ \pi^b.
 \]
 Then, if $(j_X)_i : V_i \otimes V_i^* \to X \otimes X^*$ denotes the $i$th component of $j_X$ for every $X \in \calC$ and $i \in I$, we have
 \begin{align*}
  (f \otimes \id_{X^*}) \circ (j_X)_i 
  &= (f_i \otimes \id_{X^*}) \circ (j_X)_i 
  = \sum_{a=1}^{n(X,i)} \left( (f \circ \iota_a) \otimes (\pi^a)^* \right) \\*
  &= \sum_{a=1}^{n(X,i)} \sum_{b=1}^{n(Y,i)} \sum_{c=1}^{n(X,i)} (f_i)^b_c \left( (\kappa_b \circ \pi^c \circ \iota_a) \otimes (\pi^a)^* \right) \\*
  &= \sum_{b=1}^{n(Y,i)} \sum_{c=1}^{n(X,i)} (f_i)^b_c \left( \kappa_b \otimes (\pi^c)^* \right) \\*
  &= \sum_{a=1}^{n(Y,i)} \sum_{b=1}^{n(Y,i)} \sum_{c=1}^{n(X,i)} (f_i)^b_c \left( \kappa_a \otimes ((\pi^c)^* \circ \kappa_b^* \circ (\rho^a)^*) \right) \\*
  &= \sum_{a=1}^{n(Y,i)} \left( \kappa_a \otimes (f_i^* \circ (\rho^a)^*) \right) 
  = (\id_Y \otimes f_i^*) \circ (j_Y)_i \\*
  &= (\id_Y \otimes f^*) \circ (j_Y)_i.
 \end{align*}
 
 Next, we claim that
 \[
  j_X \circ e_U = d_X
 \]
 for every object $X \in \calC$. Indeed, we have
 \begin{align*}
  j_X \circ e_U = \sum_{i \in I} \sum_{a=1}^{x_i} \left( \iota_a \otimes (\pi^a)^* \right) \circ d_{V_i} &= \sum_{i \in I} \sum_{a=1}^{x_i} \left( (\iota_a \circ \pi^a) \otimes \id_{X^*} \right) \circ d_X 
  = d_X,
 \end{align*}
 where the second equality follows from the fact that $\{ d_X : U \to X \otimes X^* \}_{X \in \calC}$ is a dinatural family.
 
 Finally, we claim that $e_U$ is unique. Indeed, consider morphisms $f_i : U \to V_i \otimes V_i^*$ for $i \in I$ satisfying
 \[
  \sum_{i \in I} (j_X)_i \circ f_i = d_X
 \]
 for every object $X \in \calC$. Since $(j_{V_{i_0}})_i = \delta_{i,i_0} \id_{V_{i_0} \otimes V_{i_0}^*}$ for any fixed $i_0 \in I$, we have
 \[
  f_{i_0} = \sum_{i \in I} (j_{V_{i_0}})_i \circ f_i = d_{V_{i_0}}. \qedhere
 \]
\end{proof}

\begin{proposition}\label{P:end_H-mod}
 If $\calC = \mods{H}$ is the category of finite-dimensional representations of a finite-dimensional Hopf algebra $H$, then
 \[
  \calE \cong \ad,
 \]
 and for every $X \in \calC$ the structure morphism $j_X : \ad \to X \otimes X^*$ is given by
 \[
  j_X(x) = \sum_{a=1}^n (x \cdot v_a) \otimes \phi^a,
 \]
 where $\{ v_a \in X | 1 \leqs a \leqs n \}$ and $\{ \phi^a \in X^* | 1 \leqs a \leqs n \}$ are dual bases. Furthermore, if $U \in \calC$ is a dinatural transformation of target $\calK$ with structure morphisms $d_X : U \to X \otimes X^*$, then the unique $H$-module morphism $e_U : U \to \calE$ making \eqref{E:univ_prop_end} into a commutative diagram is given by
 \[
  e_U := (\id_H \otimes \eta^*) \circ d_H,
 \]
 where $H$ denotes the regular representation and $\eta : \Bbbk \to H$ denotes the unit of $H$. 
\end{proposition}

\begin{proof}
 The fact that $j_X$ is well-defined is a direct consequence of a standard linear algebra argument. Indeed, for any other pair of dual bases
 \[
  \{ w_a \in X | 1 \leqs a \leqs n \}, \qquad 
  \{ \psi^a \in X^* | 1 \leqs a \leqs n \}
 \]
 there exist scalars $m^a_b \in \Bbbk$ for $1 \leqs a,b \leqs n$ satisfying
 \[
  w_b = \sum_{a=1}^n m^a_b v_a, \qquad
  \phi^a = \sum_{b=1}^n m^a_b \psi^b,
 \]
 which means
 \[
  \sum_{b=1}^n (x \cdot w_b) \otimes \psi^b = \sum_{a,b=1}^n m^a_b (x \cdot v_a) \otimes \psi^b = \sum_{a=1}^n (x \cdot v_a) \otimes \phi^a
 \]
 for every $x \in H$. 
 
 The fact that $j_X$ is an intertwiner follows from
 \begin{align*}
  j_X(\ad_x(y)) &= \sum_{a=1}^n (x_{(1)}yS(x_{(2)}) \cdot v_a) \otimes \phi^a 
  = \sum_{a=1}^n (x_{(1)}y \cdot v_a) \otimes (x_{(2)} \cdot \phi^a) 
  = x \cdot j_X(y)
 \end{align*}
 for all $x,y \in H$, where the second equality follows from the fact that
 \[
  (x \cdot \phi)(v) = \phi(S(x) \cdot v)
 \]
 for all $x \in H$, $v \in X$, and $\phi \in X^*$.
  
 Next, we claim that 
 \[
  (f \otimes \id_{X^*}) \circ j_X  = (\id_Y \otimes f^*) \circ j_Y
 \]
 for every $H$-module morphism $f : X \to Y$. Indeed, for every $x \in H$ we have
 \begin{align*}
  (f \otimes \id_{X^*})(j_X(x)) &= \sum_{a=1}^n f(x \cdot v_a) \otimes \phi^a 
  = \sum_{a=1}^n x \cdot f(v_a) \otimes \phi^a 
  = \sum_{a=1}^n x \cdot v_a \otimes f^*(\phi^a) \\*
  &= (\id_X \otimes f^*)(j_X(x)).
 \end{align*}

 Next, we claim that $e_U$ is an intertwiner. Indeed, if for every $u \in U$ we adopt the notation
 \[
  d_H(u) = \sum_{a=1}^k d_H(u)_a \otimes d_H(u)^a
 \]
 with $d_H(u)_a \in H$ and $d_H(u)^a \in H^*$ for every $1 \leqs a \leqs k$, then
 \[
  e_U(u) = \sum_{a=1}^k d_H(u)^a(1) d_H(u)_a. 
 \]
 This means
 \begin{align*}
  e_U(x \cdot u) &= \sum_{a=1}^k d_H(x \cdot u)^a(1) d_H(x \cdot u)_a 
  = \sum_{a=1}^k d_H(u)^a(S(x_{(2)})1) x_{(1)}d_H(u)_a \\*
  &= \sum_{a=1}^k d_H(u)^a(1S(x_{(2)})) x_{(1)}d_H(u)_a 
  = \sum_{a=1}^k d_H(u)^a(1) x_{(1)}d_H(u)_aS(x_{(2)}) \\*
  &= \sum_{a=1}^k d_H(u)^a(1) \ad_x(d_H(u)_a) 
  = \ad_x(e_U(u))
 \end{align*}
 for every $x \in H$, where the second equality follows from the fact that $d_H$ is an intertwiner, and where the fourth equality follows from the fact that it is dinatural.
 
 Next, we claim that
 \[
  j_X \circ e_U = d_X
 \]
 for every $H$-module $X$. Indeed, if for every $u \in U$ we adopt the notation
 \[
  d_X(u) = \sum_{a=1}^\ell d_X(u)_a \otimes d_X(u)^a
 \]
 with $d_X(u)_a \in X$ and $d_X(u)^a \in X^*$ for every $1 \leqs a \leqs \ell$, then
 \begin{align*}
  j_X(e_U(u)) &= \sum_{a=1}^k \sum_{b=1}^n d_H(u)^a(1) (d_H(u)_a \cdot v_b) \otimes \phi^b \\*
  &= \sum_{a=1}^\ell \sum_{b=1}^n d_X(u)^a(1 \cdot v_b) d_X(u)_a \otimes \phi^b 
  = \sum_{a=1}^\ell \sum_{b=1}^n d_X(u)^a(v_b) d_X(u)_a \otimes \phi^b \\*
  &= \sum_{a=1}^\ell d_X(u)_a \otimes d_X(u)^a 
  = d_X(u),
 \end{align*}
 where the second equality follows from the fact that $\{ d_X : U \to X \otimes X^* \}_{X \in \calC}$ is a dinatural family, and that $x \mapsto x \cdot v$ defines an $H$-module morphism from $H$ to $X$ for every $v \in X$.
 
 Finally, we claim that $e_U$ is unique. Indeed, consider an $H$-module morphism $f : U \to \ad$ satisfying
 \[
  j_X \circ f = d_X
 \]
 for every $H$-module $X$. If we fix $X = H$ then, since 
 \[
  x = \sum_{a=1}^n \phi^a(x) v_a = \sum_{a=1}^n \phi^a(x \cdot 1) v_a = \sum_{a=1}^n \phi^a(1) x \cdot v_a = (\id_H \otimes \eta^*)(j_H(x))
 \]
 for every $x \in H$, we have 
 \[
  f(x) = (\id_H \otimes \eta^*)(j_H(f(x)))
  = (\id_H \otimes \eta^*)(d_H(x))
  = e_U(x) \qedhere
 \]
\end{proof}

Just like the coend $\calL$, the end $\calE$ admits the structure of a braided Hopf algebra in $\calC$. When $\calC = \mods{H}$ is the category of finite-dimensional representations of a finite-dimensional ribbon Hopf algebra $H$, this can be expressed in terms of the structure morphisms 
\begin{align*}
 \mu &: H \otimes H \to H, &
 \eta &: \Bbbk \to H, \\*
 \Delta &: H \to H \otimes H, &
 \epsilon &: H \to \Bbbk, \\*
 S &: H \to H, &
 &
\end{align*}
and of the R-matrix $R = R' \otimes R'' \in H \otimes H$, which determines the Drinfeld element $u = S(R'')R' \in H$. This was first explicitly spelled out by Majid \cite{M91b}, who referred to the resulting braided Hopf algebra as the \textit{transmutation} of $H$, denoted $\bar{H}$. Thanks to Prop.~\ref{P:end_H-mod}, as an object $\bar{H}$ is given by the adjoint representation. The braided Hopf algebra structure on $\bar{H}$ is given, for all $x,y \in \bar{H}$, by
\begin{align*}
 \bar{\mu}(x \otimes y) &= \mu(x \otimes y) = xy, &
 \bar{\eta}(1) &= 1, \\
 \bar{\Delta}(x) &= \ad_{R''}(x_{(2)}) \otimes R'x_{(1)} &
 \bar{\varepsilon}(x) &= \varepsilon(x), \\*
 &= x_{(1)}S(R'') \otimes \ad_{R'}(x_{(2)}), & & \\
 \bar{S}(x) &= R''u^{-1}S(x)S(R') &
 & \\*
 &= u^{-1}S(R'')S(x)R', &
 &
\end{align*}
compare with \cite[Sec.~3.4]{L94}.

\end{document}